\documentclass[a4paper,10pt]{article}

\usepackage{indentfirst,amsmath,amsfonts,amstext,amssymb,amscd,bezier,amsthm}
\usepackage[latin1]{inputenc}
\usepackage[dvips]{graphicx}
\usepackage[latin1]{inputenc}
\usepackage{setspace}
\usepackage{graphicx}
\usepackage{tabularx}
\usepackage{longtable}
\usepackage[usenames,dvipsnames]{pstricks}
\usepackage{epsfig}
\usepackage{pst-grad}
\usepackage{pst-plot}
\usepackage[all]{xy}
\usepackage{hyperref}
\usepackage[T1]{fontenc}
\usepackage{ae,aecompl}
\usepackage{indentfirst}
\usepackage{xspace}
\usepackage{color}
\usepackage{psfrag}
\usepackage{mathrsfs}
\usepackage{upgreek}                            
\usepackage{makeidx}                            
\usepackage{fancyhdr}
\usepackage{fancybox}
\usepackage[rm]{titlesec}
\usepackage{float}   
\usepackage{enumerate}     

\input xy
\xyoption{all}
\newcommand\numberthis{\addtocounter{equation}{1}\tag{\theequation}}

\newcommand{\p}{\mathbb{P}}

\newcommand{\F}{\mathbb{F}}

\newcommand{\lra}{\longrightarrow}

\newcommand{\Lra}{\Longrightarrow}

\newcommand{\ff}{\mathcal{F}}
\newcommand{\hh}{\mathcal{H}}
\newcommand{\qq}{\mathcal{Q}}
\newcommand{\G}{\mathcal{G}}
\newcommand{\fq}{\mathbb{F}_q}

\newcommand{\cc}{\mathcal C}

\newcommand{\fqc}{\overline{\mathbb{F}}_q}

\theoremstyle{plain}
\newtheorem{thm}{Theorem}[section]

\newtheorem{prop}[thm]{Proposition}
\newtheorem{lem}[thm]{Lemma}

\newtheorem{rem}[thm]{Remark}
\newtheorem{ex}[thm]{Example}

 \font\numberfont= pzcmi scaled
3000
\titleformat{\chapter}[display]
  {\normalfont\Large 
  }
  {
   \filright
   \rule[32pt]{.7\linewidth}{4pt}
   \hspace{-8pt}
   \shadowbox{
   \begin{minipage}{.15\linewidth}
     \begin{center}
          \textsl{\bf {\large \chaptertitlename}}\\
       \vspace{1ex}
       {\bf {\numberfont \thechapter}}\\
       \vspace{1ex}
     \end{center}
   \end{minipage}}
  }
  {-10pt}
  {\filcenter
           \sl
           \bf
              \Huge
     }
  [\vspace{-1cm}\hfill\rule{.8\textwidth}{0.5pt}\\
\vskip-2.8ex\hfill\rule{.7\textwidth}{4pt}\vspace*{-1ex}]
\titlespacing{\chapter}{0pt}{*4}{*1}

\titleformat{\section}[block]
{\normalfont\bfseries} {\thesection}{0.5em}{}

\titleformat{\subsection}[block]
{\normalfont\large\bfseries} {\thesubsection}{0.5em}{}

\makeindex                                      
\makeatletter
\numberwithin{equation}{section}

\begin{document}

\title{Frobenius nonclassicality of  Fermat curves with respect to cubics}

\author{\textbf{Nazar Arakelian} \\
 \small{IMECC, Universidade Estadual de Campinas, Campinas, Brazil} \\
 \textbf{Herivelto Borges}\\
 \small{ICMC, Universidade de S\~ao Paulo, S\~ao Carlos, Brazil}}

\maketitle
 \begin{abstract}
 For Fermat curves  $\ff:aX^n+bY^n=Z^n$ defined over $\fq$,  we establish necessary and sufficient conditions for $\ff$ to be $\fq$-Frobenius nonclassical with respect to the linear system of plane cubics. In the  $\fq$-Frobenius nonclassical cases, we determine explicit formulas for the  number $N_q(\ff)$  of $\fq$-rational points on $\ff$.
 For the remaining Fermat curves, nice upper bounds for  $N_q(\ff)$ are immediately given by the St\"ohr-Voloch Theory.
 \end{abstract}


\section {Introduction}\label{intro}

Let  $\fq$ be the finite field with $q=p^h$ elements. For an irreducible   Fermat curve
\begin{equation}\label{Fermat}
\ff:aX^n+bY^n=Z^n 
\end{equation}
defined over $\fq$, let $N_q(\ff)$ denote its  number of $\fq$-rational points. The celebrated   Hasse-Weil Theorem   gives
\begin{equation}\label{H-W}
|N_q(\ff)-(q+1)| \leq (n-1)(n-2)\sqrt{q}.
\end{equation}
In 1986,  St\"ohr  and Voloch introduced a new technique to bound the number of rational points
on curves over finite fields \cite{SV} . Their method uses some data collected from  embeddings of the curve in  projective spaces, and  in many circumstances  it gives improvements upon the Hasse-Weil bound. 

For example,  let $\ff$  be a Fermat  curve  as given in \eqref{Fermat}.  For   $s \in \{1,\ldots,n-1\}$, consider the linear system $\Sigma_s$  of all curves in $\p^2(\fqc)$ of degree $s$. Associated to $\Sigma_s$, there exists  a sequence of non-negative integers $(\nu_0,\ldots,\nu_{M-1})$ (depending   on $\ff$, $q$ and $s$) such that $0=\nu_0<\cdots<\nu_{M-1}$ and $M={s+2 \choose 2}-1$. If $\nu_i=i$ for all $i=0,\ldots,M-1$, then the curve $\ff$ is called $\fq$-Frobenius classical with respect to $\Sigma_s$.  Otherwise, $\ff$ is called $\fq$-Frobenius nonclassical. Together with \cite[Proposition 2.4]{SV} and some remarks in Section 3 of \cite{SV}, the St\"ohr-Voloch Theorem \cite[Theorem 2.13]{SV} applied to $\Sigma_s$  gives
\begin{equation}\label{SV-s}
N_q(\ff) \leq \displaystyle\frac{n(n-3)(\nu_1+\cdots +\nu_{M-1})+sn(q+M)}{M}-\sum_{P \in \ff}\frac{A(P)}{M},
\end{equation}
where
\begin{equation*}
A(P)= \begin{cases} 
\sum_{l=1}^{M}(j_l-\nu_{l-1})-M,  \text{ if }\quad P \text{ is  an } \fq\text{-rational  point}\\
\sum_{l=1}^{M-1}(j_l-\nu_{l}),  \text{ otherwise },
\end{cases} 
  \end{equation*}
and  $0=j_0<j_1 <\cdots <j_M$ are the $(\Sigma_s,P)$-orders.

If $\ff$ is $\fq$-Frobenius classical with respect to $\Sigma_s$, then bound (\ref{SV-s}) reads (cf. \cite[Theorem 1]{GV2})
\begin{equation}\label{SV-sc}
N_q(\ff) \leq \displaystyle\frac{n(n-3)(M-1)}{2} +\displaystyle\frac{sn(q+M)}{M}-\displaystyle\frac{3nA+dB}{M},
\end{equation}
where $B=sn-M$,
\begin{equation}\nonumber
A=\frac{1}{6}\left((n-s-1)s(s-1)(s+4)+\frac{s(s-1)(s-2)(s+5)}{4}\right),
\end{equation}
and $d$ is the number of $\fq$-rational points $P=(u:v:w) \in \ff$ for which $uvw=0$.

Note that in the $\fq$-Frobenius nonclassical case we have
$
\nu_1 +\cdots+\nu_{M-1}>M(M-1)/2.
$
Thus  bound \eqref{SV-sc} may not hold. This suggests that  such curves are  likely to have many $\fq$-rational points.

Characterizing  $\fq$-Frobenius nonclassical  curves  may offer a two-fold benefit. If we can identify the  $\fq$-Frobenius nonclassical curves, then we are left with a class of curves for which  a better upper bound (inequality \eqref{SV-sc}) for the number of $\fq$-rational points holds. At the same time, the $\fq$-Frobenius nonclassical curves provide a potential source of curves with many such points.

In 1988, Garcia and Voloch characterized the  $\fq$-Frobenius nonclassical  Fermat curves with respect to  $\Sigma_s$, in the cases $s=1$ and $s=2$ \cite{GV2}. It is not hard to check that if $n$ is small with respect to $q$, then bound \eqref{SV-sc}  becomes stronger for larger values of $s$. Thus the characterization of the $\fq$-Frobenius nonclassicality for other values of $s$ is highly desirable, and it is naturally challenging. 

In this manuscript, we establish the result for $s=3$. That is,  we  characterize  the $\fq$-Frobenius nonclassical Fermat curves with respect to the linear systems of plane cubics. Our main result is the following:
 \begin{thm}\label{frobenius}
Let $\ff:aX^{n}+bY^{n}=Z^{n}$ be an irreducible Fermat curve defined over $\fq$, where $q=p^h$, $p>11$, and $n>3$. Suppose that $\ff$ is $\fq$-Frobenius classical with respect to $\Sigma_2$. Then the curve $\ff$ is $\fq$-Frobenius nonclassical with respect to $\Sigma_3$ if and only if one of the following holds:
\begin{enumerate}[\rm(i)]

\item $p|n-3$ and 
$
n=\frac{3(p^h-1)}{p^r-1}
$
for some $r<h$ such that  $r|h$ and $a,b \in \F_{p^r}$.
\item $p|3n-1$ and $n=\frac{p^h-1}{3(p^r-1)}$ for some $r<h$ such that  $r|h$ and $a^3,b^3 \in \F_{p^r}$.
\end{enumerate}
\end{thm}

Some   techniques applied here can be carried over to larger values of $s$, and thereby shed some light on  the solution of this problem for 
the general linear system $\Sigma_s$.
   
The paper is organized as follows. Section 2 sets some  notation and recalls  results from the St\"ohr-Voloch Theory. Section 3 provides necessary and sufficient conditions for the curve $\ff$ to be nonclassical with respect to $\Sigma_3$. In particular, it answers a question raised by Garcia and Voloch in \cite{GV2}.   Section 4 presents a sequence of  additional   results culminating in  the proof of Theorem \ref{frobenius}.  Finally, Section 5 determines  the number of  $\fq$-rational points on  curves given by Theorem \ref{frobenius} and provides some examples. In the paper's appendix,  we  prove the irredubility of some   low-degree curves, and include a proof  for a case of Frobenius nonclassicality w.r.t. $\Sigma_2$ which was apparently  overlooked in  \cite{GV2}. An unpublished but very useful result, due to M. Homma
and S. J. Kim,  is also included in the appendix.

\text{}\\

\textbf{Notation}

\text{}\\
Hereafter, we  use the following notation: 

\begin{itemize}
\item $\fq$ is the finite field with $q=p^h$ elements, with $h \geq 1$, for a prime integer $p$.
\item $\fqc$ is the algebraic closure of $\fq$.
\item Given an irreducible curve $\cc$ over $\fq$ and an algebraic extension $\mathbb{H}$ of $\fq$, the function field of $\cc$ over $\mathbb{H}$ is denoted by $\mathbb{H}(\cc)$.
\item For a curve $\cc$ and $r>0$, the set of its $\F_{q^r}$-rational points is denoted by $\cc(\F_{q^r})$.
\item $N_{q^r}(\cc)$ is the number of $\F_{q^r}$-rational points of the curve $\cc$.
\item For a nonsingular point $P \in \cc$, the discrete valuation at $P$ is denoted by $v_P$.
\item For two plane curves $\cc_1$ and $\cc_2$ , the intersection multiplicity of $\cc_1$ and $\cc_2$  at the point $P$ is denoted by $I(P,\cc_1 \cap \cc_2)$.
\item Given $g \in \fqc(\cc)$, $t$ a separating variable of $\fqc(\cc)$, and $r \geq 0$, the $r$-th Hasse derivative of $g$ with respect to $t$ is denoted by  $D_t^{(r)}g$.
\end{itemize}


\section{Preliminaries}\label{prelim}

Let us start  by recalling  the main results of \cite{GV2} and \cite{SV}. For $n>3$, consider an irreducible  Fermat curve
\begin{equation}
\ff:aX^n+bY^n=Z^n
\end{equation}
 defined over $\fq$.
 For each $s \in \{1,\ldots,n-1\}$, denote by $\Sigma_s$ the linear system of all projective plane curves of degree $s$. For any point  $P \in \ff$,  an integer $j:=j(P)$ is called a $(\Sigma_s,P)$-order if there exists a plane curve of degree $s$, say $\mathcal{C}_P$, such that $I(P,\ff \cap \mathcal{C}_P)=j$.  From the discussion  in \cite[Section 1]{SV}, it follows that there exist exactly $M+1$ $(\Sigma_s,P)$-orders
$$
j_0(P)<j_1(P)< \cdots <j_M(P),
$$
where $M={s+2 \choose 2}-1$. The sequence $(j_0(P),j_1(P), \ldots ,j_M(P))$ is called $(\Sigma_s,P)$-order sequence.
 Note that  $j_0(P)=0$ and  $j_1(P)=1$ for all $P \in \ff $. 
 Moreover,  there exists a unique curve $\hh_P$ of degree $s$, called $s$-osculating curve of $\ff$ at $P$, such that $I(P,\ff \cap \hh_P)=j_M(P)$ \cite[Theorem 1.1]{SV}. All but finitely many points of $\ff$ have the same order sequence, denoted by $(\varepsilon_0,\ldots,\varepsilon_M)$. This sequence is called order sequence of $\ff$ with respect to $\Sigma_s$, and  the integers $\varepsilon_i$ are called $\Sigma_s$-orders.

Let $\fqc(\ff)=\fqc(x,y)$ be the function field of $\ff$, defined by   $ax^n+by^n=1$. To each linear series $\Sigma_s$, there corresponds a morphism
\begin{equation}\label{veronese}
\phi_s=(\ldots:x^iy^j:\ldots):\ff \lra \p^M(\fqc),
\end{equation}
where $i+j \leq s$, called the $s$-Veronese morphism. Let  $t$ be a separating variable of $\fqc(\ff)$ and $D_t^{(i)}$ denote the $i$-th Hasse derivative with respect to  $t$. The $\Sigma_s$-orders of $\ff$  can also be defined as  the minimal sequence with respect to the lexicographic order, for which the function
\begin{equation}\nonumber
\det\left( D_t^{(\varepsilon_k)} (x^iy^j) \right)_{\tiny{\begin{array}{c}0 \leq k \leq M, \\[0.1mm]
0 \leq i+j \leq s \end{array}}}
\end{equation}
is nonvanishing. Moreover, this minimality implies that
$\varepsilon_i \leq j_i(P)$  for all $i \in \{1, \ldots ,M \}$  and $P \in \ff$. The curve $\ff$ is called classical with respect to $\Sigma_s$ (or $\phi_s$) if the sequence $(\varepsilon_0,\ldots,\varepsilon_M)$ is $(0,\ldots,M)$. Otherwise, it is called nonclassical. 

The following result concerning $\Sigma_s$-orders is proved  in \cite[Corollary 1.9]{SV}.

\begin{thm}\label{cor1.9}
Let $\varepsilon$ be a $\Sigma_s$-order. Then every integer $\mu$ such that 
$$
{\varepsilon \choose \mu } \not\equiv 0  \mod p
$$
is also a $\Sigma_s$-order. In particular, if $\varepsilon<p$, then $0,1,\ldots,\varepsilon-1$ are $\Sigma_s$-orders. 
\end{thm} 

The following is a significant  criterion for determining whether $\ff$ is classical (see \cite[Proposition 1.7]{SV}).

\begin{prop}\label{SV1.7}
Let $P \in \ff$ be a point with order sequence $(j_0(P),\ldots,j_M(P))$. If the integer 
$$\prod_{i>r}\frac{j_i(P)-j_r(P)}{i-r}$$ is not divisible by $p$, then $\ff$ is classical with respect to $\Sigma_s$.
\end{prop}

The following characterization of the order sequences is given in \cite[Proposition 2]{GV1}.

\begin{prop}\label{pwr}
Let $\varepsilon_0<\varepsilon_1<\cdots<\varepsilon_M$ be the orders of $\ff$ with respect to $\Sigma_s$. Suppose  $p \geq M$ and $\varepsilon_i=i$ for $i=0,1,\ldots,M-1$. If $\varepsilon_M>M$, then $\varepsilon_n$ is a power of $p$.
\end{prop}

In \cite[Theorem 3]{GV1}, Garcia and Voloch gave  a complete characterization of the nonclassical Fermat curves with respect to conics:

\begin{thm}\label{gv nc}
Suppose $p>5$. The Fermat curve $\ff$ is nonclassical with respect to $\Sigma_2$ if and only if
$$
p \text{ divides }(n-2)(n-1)(n+1)(2n-1).
$$
\end{thm}

Let us recall that there exists a sequence of non-negative integers $(\nu_0,\ldots,\nu_{M-1})$, chosen minimally in the lexicographic order, such that
\begin{equation}\label{fncl}
\left|
  \begin{array}{ccccccc}
  1 & x^q & y^q&\ldots & (x^iy^j)^q & \ldots& (y^s)^q \\
  D_t^{(\nu_0)}1 &D_t^{(\nu_0)}x  &D_t^{(\nu_0)}y& \ldots & D_t^{(\nu_0)}x^iy^j&\ldots& D_t^{(\nu_0)}y^s \\
  \vdots&\vdots&\vdots & \cdots & \vdots & \cdots &\vdots\\
  D_t^{(\nu_{M-1})}1& D_t^{(\nu_{M-1})}x &D_t^{(\nu_{M-1})}y & \cdots & D_t^{(\nu_{M-1})}x^iy^j& \ldots & D_t^{(\nu_{M-1})}y^s
  \end{array}
  \right| \neq 0,
  \end{equation}
where $t$ is a separating variable of $\fq(\ff)$  \cite[Proposition 2.1]{SV}. This sequence is called $\fq$-Frobenius sequence of $\ff$ with respect to $\Sigma_s$. It turns out that $\{\nu_0,\ldots,\nu_{M-1}\}=\{\varepsilon_0,\ldots,\varepsilon_M\}\backslash\{\varepsilon_I\}$ for some $I \in \{1,\ldots,M\}$ \cite[Proposition 2.1]{SV}. If $(\nu_0,\ldots,\nu_{M-1})=(0,\ldots,M-1)$, then the curve $\ff$ is called $\fq$-Frobenius classical with respect to $\Sigma_s$. Otherwise, it  is called $\fq$-Frobenius nonclassical.

The following  result  establishes a close relation between classicality and $\fq$-Frobenius classicality, see \cite[Remark 8.52]{HKT}.

\begin{prop}\label{fnc impl nc}
Assume  $p>M$. If $\ff$ is $\fq$-Frobenius nonclassical with respect to $\Sigma_s$, then $\ff$ is nonclassical with respect to $\Sigma_s$.
\end{prop}

The $\fq$-Frobenius map $\Phi_q$ is defined on $\ff$ by
\begin{displaymath}
\begin{array}{cccc}
\Phi_q: &\ff & \lra & \ff \\
        &(a_0:a_1:a_2) & \longmapsto & (a_0^q:a_1^q:a_2^q). 
\end{array}
\end{displaymath}
Note that by (\ref{fncl}) and \cite[Corollary 1.3]{SV}, we have that $\ff$ is $\fq$-Frobenius nonclassical with respect to $\Sigma_1$  if  and only if $\Phi_q(P)$ lies on the tangent line of $\ff$ at $P$ for all $P \in \ff$. More generally,  (\ref{fncl}) and \cite[Corollary 1.3]{SV} give the following:

\begin{prop}\label{prop1} Suppose the order sequence $(\varepsilon_0,\ldots,\varepsilon_M)$ of $\ff$  w.r.t.   $\Sigma_s$ is such that  $\varepsilon_i=i$ for $i=0,1,\ldots,M-1$.
 Let  $\Phi_q : \ff \rightarrow \ff $ be the $\fq$-Frobenius map, and for any point $P \in  \ff$, let $\hh_P$ be the s-osculating curve to $\ff$ at $P$.  Then $\ff$ is $\fq$-Frobenius nonclassical w.r.t. $\Sigma_s$ if and only if $\Phi_q(P) \in \hh_P$ for infinitely many points $P \in \ff$. 

\end{prop}

With regard to the $\fq$-Frobenius classicality of $\ff$ in the cases $s=1$ and $s=2$, the following results were proved \cite{GV2}.

\begin{thm}[Garcia-Voloch]\label{s=1}
Suppose that $p>2$ and let $q=p^h$. Then $\ff$ is $\fq$-Frobenius nonclassical with respect to $\Sigma_1$ if and only if $n=(q-1)/(p^r-1)$ for some integer $r<h$ with $r|h$ and $a,b \in \F_{p^r}$.
\end{thm}

\begin{thm}[Garcia-Voloch]\label{s=2}
Suppose that $p>5$ and let $q=p^h$. Then $\ff$ is $\fq$-Frobenius nonclassical with respect to $\Sigma_2$ if and only if one of the following holds.
\begin{enumerate}
\item[(i)] $p | (n-1)$.
\item[(ii)] $p | (n-2)$ and $n=\frac{2(q-1)}{p^{r}-1}$ with $r<h$ such that $r|h$ and $a,b \in \F_{p^{r}}$.
\item[(iii)] $p | (2n-1)$ and $n=\frac{q-1}{2(p^{r}-1)}$ with $r<h$ such that $r|h$  and $a^{2},b^{2} \in \F_{p^{r}}$.
\item[(iv)] $q=n+1$ and $a+b=1$.
\end{enumerate}
\end{thm}

\begin{rem}\label{missing}
Item (iv) of Theorem \ref{s=2} is a minor case that was apparently overlooked in \cite{GV2}. A  proof of it is included   in the appendix.
\end{rem}

The following result  (see \cite[Theorem 1.1]{KS} or \cite[Remark 8.109]{HKT}) will be used to compute the number of $\fq$-rational points of certain Fermat curves in Section 5.

\begin{thm}[Korchm\'aros-Sz\"onyi]\label{ferko}
Let $\ff:X^n+Y^n+Z^n=0$ and $q=p^r$. Suppose that $n$ divides $\frac{q^m-1}{q-1}$, where $m>1$. Let $t$ be defined by $q \equiv t \mod \frac{q^m-1}{n(q-1)}$ and $0<t<\frac{q^m-1}{n(q-1)}$. If $l=\gcd\left(\frac{q^m-1}{n(q-1)},t+1\right)$, then
$$
N_{q^m}(\ff)=3n+n^2(q-2)+n^2(l-1)(l-2)
$$
provided that
$$
p>\left(\frac{2}{\sqrt[t+1]{\sin\left(\frac{n(q-1)\pi}{2(q^m-1)}\right)}}+1\right)^{(t-1)\left(\frac{q^m-1}{n(q-1)}-l\right)}.
$$
\end{thm}


\section{Classicality of $\ff$ with respect to cubics}\label{class}

Let us recall that $\ff:aX^n+bY^n=Z^n$ is an irreducible curve defined over $\fq$. Based on  Proposition \ref{fnc impl nc},  the study of $\fq$-Frobenius nonclassicality of $\ff$, with respect to $\Sigma_s$, can benefit directly from the study  of  nonclassicality of $\ff$. In this section, we  establish necessary and sufficient conditions for $\ff$ to be nonclassical with respect to $\Sigma_3$.

\begin{rem}\label{pd}
As mentioned in  Section \ref{prelim}, for any point $P \in \ff$, the number of distinct $(\Sigma_s,P)$-orders is ${s+2 \choose 2}$. In particular, there exist $10$ distinct $(\Sigma_3,P)$-orders.   
\end{rem}

Note that if the curve $\ff$ is nonclassical with respect to lines, then it is also nonclassical with respect to conics. Indeed assume that   the order sequence of $\ff$ for $\Sigma_1$ is $(0,1,\varepsilon>2)$. Thus, considering  the conics given by the unions of two of these lines, we have that $0,1,2,\varepsilon,\varepsilon+1$, and $2\varepsilon$ comprise the 6 distinct  $\Sigma_2$-orders (cf. Remark \ref{pd}). Similarly, one can see that if $\ff$ is nonclassical with respect to lines, then it is also nonclassical  with respect to cubics.

Now assume that $p>7$ and that $\ff$ is classical with respect to $\Sigma_1$ but nonclassical with respect to $\Sigma_2$. Then, by Proposition \ref{pwr}, the order sequence of $\ff$ with respect to $\Sigma_2$ is $(0,1,2,3,4,p^r)$, for some integer $r>0$. Considering all  possible unions of a conic and  a line, we have that the order sequence of $\ff$ with respect to $\Sigma_3$ is $(0,1,2,3,4,5,6,p^r,p^r+1,p^r+2)$. Therefore, $\ff$ is nonclassical with respect to $\Sigma_3$. The next lemma summarizes the above discussion.

\begin{lem}\label{con imp cub}
Suppose  $p>7$. If $\ff$ is nonclassical with respect to either $\Sigma_1$ or $\Sigma_2$, then $\ff$ is nonclassical with respect to $\Sigma_3$.
\end{lem}

The following  result, which  will be a critical factor in our approach, extends  \cite[Lemma 1.3.8]{Na}. Its proof, which   was provided by M. Homma and S. J. Kim in a private communication, can be found in the appendix. 

\begin{lem}[Homma-Kim]\label{trans} Let $S$ be a surface defined over  an algebraically closed field $K$, and let  $P \in S$ be a nonsingular point.
If $C$, $D_1$ and $D_2$ are  effective divisors, of which  no two  have a common component, and $P$ is a  nonsingular point of $C$, 
then  $$I(D_1. D_2,P) \geq min \{I(C . D_1,P),I(C. D_2,P)\}.$$
\end{lem}

Since a curve's classicality  is a geometric property, for this section it is assumed that $a=b=1$. We begin with some preliminary results.

\begin{lem}\label{simp} 
Assume $p>7$. Let $\fqc(x,y)$ be the function field of $\mathcal{F}$, and $P=(u:v:1) \in \mathcal{F}$ be a generic point. Suppose that there exists a polynomial $G(X,Y)=\sum a_{ij}(x,y)^pX^iY^j \in \fqc[x,y][X,Y]$  of degree $d\geq 3$ such that $G(x,y)=0$. For $G_P(X,Y):=\sum a_{ij}(u,v)^pX^iY^j \in \fqc [X,Y]$,  the following holds:
\begin{enumerate}[\rm(a)]

\item If $G_P(X,Y)$ is irreducible of degree $d=3$, then $\mathcal{F}$ is nonclassical with respect to $\Sigma_3$ and the curve $\G_P: G_P(X,Y)=0$
is the osculating cubic to  $\ff$ at $P$.

\item 
If  the curve $\G_P: G_P(X,Y)=0$ is such that
$I(P,\G_P \cap \mathcal{C})<p$ for any cubic  $\mathcal{C}$, then  $\mathcal{F}$ is classical with respect to $\Sigma_3$.

\end{enumerate}

\end{lem}

\begin{proof}
Let $\G_P$ be the curve defined by $G_P(X,Y)=0$. Since
\begin{eqnarray}
G_P(x,y)&=&G_P(x,y)-G(x,y) \nonumber\\
      &=&\sum (a_{ij}(u,v)-a_{ij}(x,y))^px^iy^j, \nonumber
\end{eqnarray}
it follows that $v_P(G_P(x,y)) \geq p$, that is, 
\begin{equation}\label{domi}
I(P,\ff \cap \G_P) \geq p.
\end{equation}

Let $\hh_P$ be the osculating cubic to $\ff$ at $P$. For assertion (a),  note that $\deg(\G_P)=3$ and inequality \eqref{domi} imply $(P,\ff \cap \hh_P) \geq p$, and then Lemma \ref{trans} gives
$$
I(P,\hh_P \cap \G_P) \geq p >9=\deg(\hh_P) \cdot \deg(\G_P).
$$
Thus, by B\'ezout's theorem, $\G_P$ and $\hh_P$ must be the same curve, and since $P$ is generic the result follows.  Assertion (b)  follows directly from Lemma \ref{trans} and from the fact that  the   nonclassicality of $\ff$ with respect to $\Sigma_3$ implies $I(P,\ff \cap \hh_P) \geq p$ (c.f. Theorem \ref{cor1.9}). 

\end{proof}

\begin{rem}\label{lemalt} Note that   if $G_P(X,Y)$ is irreducible of degree $<p/3$, then, by B\'ezout's theorem, the conditions on Lemma \ref{simp}(b) are fulfilled, i.e.,
  $\mathcal{F}$ is classical with respect to $\Sigma_3$.
\end{rem}

\begin{lem}\label{classic}
If  $p>11$ divides  $(n+2)(2n+1)(2n-3)(3n-2)$, then $\ff$ is classical with respect to $\Sigma_3$.
\end{lem}
\begin{proof} We first prove the result for $p>17$.

Suppose  $p|n+2$, and  let $m,r>0$ be  integers such that  $n=mp^r-2$ and $p\nmid m$. It follows from $x^n+y^n=1$ that $(x^n+y^n-1)x^2y^2 = 0$, and then
\begin{eqnarray}
(x^{mp^r})y^2+(y^{mp^r})x^2-x^2y^2 = 0 . \nonumber
\end{eqnarray}
Consider $P=(u:v:1) \in \ff$, with $uv \neq 0$, and set $\alpha=v^{mp^r}$ and $\beta=u^{mp^r}$. By Lemma \ref{polirr}, the curve $\G_1:\alpha X^2Z^2+\beta Y^2Z^2 -X^2Y^2=0$ is irreducible. Therefore, Remark \ref{lemalt} implies that $\ff$ is classical w.r.t. $\Sigma_3$.
 To address the cases $p|2n+1$ and $p|2n-3$, note that
 \begin{eqnarray}
x^{n}+y^{n}-1=0 & \Lra & (x^{n}+y^{n}-1)(x^{n}+y^{n}+1)((x^{n}-y^{n})^2-1) =0\nonumber \\ 
                    & \Lra & x^{4n}-2x^{2n}y^{2n}-2x^{2n}+y^{4n}-2y^{2n}+1=0\nonumber        
\end{eqnarray}
  yields
\begin{equation}\label{eq11}
x^{2(2n+1)}y^2-2x^{2n+1}y^{2n+1}xy-2x^{2n+1}xy^2+y^{2(2n+1)}x^2-2y^{2n+1}x^2y+x^2y^2=0 \\
 \end{equation}
and 
\begin{equation}\label{eq12}
x^{2(2n-3)}x^6-2x^{2n-3}y^{2n-3}x^3y^3-2x^{2n-3}x^3+y^{2(2n-3)}y^6-2y^{2n-3}y^3+1=0.
\end{equation}

If $p|2n+1$, we consider  integers  $m,r >0$ such that $2n+1=mp^r$ and $p\nmid m$. Likewise,  we write $2n-3=mp^r$ for  the case $p|2n-3$. Therefore \eqref{eq11} and \eqref{eq12} can be written as 

\begin{equation}\label{eq1}
(x^{2m})^{p^r}y^2-2(x^my^m)^{p^r}xy-2(x^m)^{p^r}xy^2+(y^{2m})^{p^r}x^2-2(y^m)^{p^r}x^2y+x^2y^2=0, \\
 \end{equation}
and 
\begin{equation}\label{eq2}
(x^{2m})^{p^r}x^6+(y^{2m})^{p^r}y^6-2(x^{m})^{p^r}x^3-2(y^{m})^{p^r}y^3-2(x^{m}y^m)^{p^r}x^3y^3+1=0,
\end{equation}
respectively. In either case, we consider  $P=(u:v:1) \in \ff$ such that $uv \neq 0$ and define $\alpha=u^{mp^r}$ and $\beta=v^{mp^r}$. The  above equations give rise to the curves
\begin{equation}\nonumber
\G_2:\alpha^2Y^2Z^2-2\alpha\beta XYZ^2-2\alpha XY^2Z+\beta^2X^2Z^2-2\beta X^2YZ+X^2Y^2=0.
\end{equation}
and 
\begin{equation}\nonumber
\G_3: \alpha^2X^6+\beta^2Y^6+Z^6-2(\alpha X^3Z^3+\beta Y^3Z^3 +\alpha \beta X^3Y^3)=0.
\end{equation}
After scaling coordinates, it follows from  Lemma \ref{polirr} that  these curves  are irreducible. Thus Remark \ref{lemalt} implies that  $\ff$ is classical w.r.t. $\Sigma_3$.

Finally, let us assume $p|3n-2$ and consider integers $m,r>0$ such that $3n=mp^r+2$ and $p \nmid m$. Thus
\begin{eqnarray}
1=x^n+y^n & \Lra & 1=(x^n+y^n)^3 \Lra 1=x^{3n}+y^{3n}+3x^ny^n \nonumber \\ 
          & \Lra & -27x^{3n}y^{3n}= (x^{3n}+y^{3n}-1)^3 \nonumber \\
            & \Lra & -27(x^{m}y^{m})^{p^r}x^2y^2=  ((x^{m})^{p^r}x^2+(y^{m})^{p^r}y^2-1)^3. \nonumber 
\end{eqnarray}

Similarly to the previous cases, the latter equation gives rise to an irreducible curve (cf. Lemma \ref{polirr} )
\begin{eqnarray}
\G_4: (\alpha X^2+ \beta Y^2-Z^2)^3+27\alpha \beta X^2Y^2Z^2=0, \nonumber
\end{eqnarray}
 and  then Remark \ref{lemalt}  finishes the proof.

In all prior cases, since $p>17$ and  $\deg({\mathcal{G}_i})<p/3$  for each $i=1,\dots,4$,  Remark \ref{lemalt} is sufficient to prove the classicality $\mathcal{F}$ with respect to $\Sigma_3$. To address the cases $p\in \{13,17\}$, the previous argument is slightly refined: note that using a suitable projective tranformation $(X:Y:Z)\mapsto (\lambda X:\lambda Y:Z)$, we can  always choose  a point $P_i=(u:u:1) \in \mathcal{F}$  and a cubic $\mathcal{C}_i$ such that 
 $$I(P_i,\G_i \cap \mathcal{C}_i)=I(\tilde{P_i},\tilde{\G_i} \cap \tilde{\mathcal{C}_i})\in \{10,12\},$$
  where $\tilde{\G_i}, \tilde{\mathcal{C}_i}$ and  $\tilde{P_i}$ are given by Lemma \ref{poliosc}. Since $\tilde{P_i} \in \tilde{\mathcal{G}_i}$ is a nonsingular point,  then so is $P_i\in \G_i$. Now if there is another cubic $\mathcal{C}$ such that $I(P_i,\G_i \cap \mathcal{C})\geq 10$, the by Lemma \ref{trans},  $I(P_i,\mathcal{C}_i \cap \mathcal{C})\geq 10$. This  contradicts Bezout's Theorem as  $\mathcal{C}_i\cong \tilde{C_i} $ is irreducible. Therefore, 
  $$I(P_i,\G_i \cap \mathcal{C})\leq 12<p$$
for all cubics $\mathcal{C}$, and then Lemma \ref{simp}(b)  gives the result.

\end{proof}

\begin{prop}\label{p div (n-3)(3n-1)}
If  $p>7$ divides  $(n-3)(3n-1)$, then $\ff:aX^n+bY^n=Z^n$ is nonclassical with respect to $\Sigma_3$. 
Moreover, for $P=(u:v:1) \in \ff$, $uv \neq 0$,  the osculating cubic $\hh_P$ to $\ff$ at $P$ is the irreducible curve $H_P(X,Y,Z)=0$, where
\begin{equation*}
H_P(X,Y,Z)= \begin{cases} 
au^{n-3}X^{3}+bv^{n-3}Y^{3}-Z^3, & \text{ if }\quad p \mid n-3\\
(a^3u^{3n-1}X+b^3v^{3n-1}Y-Z)^3+27a^3b^3(uv)^{3n-1}XYZ, & \text{ if }\quad p \mid 3n-1.
\end{cases} 
  \end{equation*}

\end{prop}
\begin{proof}
Suppose $p|n-3$, and let $m,r>0$ be  integers such that $n=mp^r+3$ and $p\nmid m$. Note that  for
$G(X,Y)=ax^{mp^r}X^{3}+by^{mp^r}Y^{3}-1$, we have $G(x,y)=0$. Since $G_P(X,Y):=au^{mp^r}X^{3}+bv^{mp^r}Y^{3}-1$ is  irreducible  of degree $3$, Lemma \ref{simp}(a) implies that  $\ff$ is nonclassical with respect to $\Sigma_3$ and $H_P(X,Y,Z)=0$ is the osculating cubic to $\ff$ at $P$.

For the case  $p|3n-1$,
note that $ax^n+by^n=1$ implies
\begin{eqnarray}\label{osc ge}
(ax^n+by^n)^3 &=& 1 \Lra\nonumber \\
a^3x^{3n}+b^3y^{3n}+3abx^ny^n &=&  1\Lra \nonumber \\
(a^3(x^m)^{p^r}x+b^3(y^m)^{p^r}y-1)^3 &=& -27a^3b^3(x^my^m)^{p^r}xy,
\end{eqnarray}
where $m,r>0$ are integers such that $3n=mp^r+1$ and $p \nmid m$. That is, for $G(X,Y):=(a^3(x^m)^{p^r}X+b^3(y^m)^{p^r}Y-1)^3+27a^3b^3(x^my^m)^{p^r}XY$, we have $G(x,y)=0$.
 The irreducibility of $G_P(X,Y):=(a^3u^{3n-1}X+b^3v^{3n-1}Y-1)^3+27a^3b^3(uv)^{3n-1}XY$ follows from that of $\tilde{\G}_4$ in Lemma \ref{polirr}. 
Therefore, Lemma \ref{simp}(a) gives the result.

\end{proof}

Next we present the main result of this section.

\begin{thm}\label{classicalidade}
If $p>11$ then $\ff:X^n+Y^n=Z^n$ is nonclassical with respect to $\Sigma_3$ if and only if $$p \text{ divides }(n-2)(n-1)(n+1)(2n-1)(n-3)(3n-1).$$
\end{thm}
\begin{proof}
Suppose that $\ff$ is nonclassical with respect to $\Sigma_3$. If  $P=(u:0:1) \in \ff(\fqc)$, and $\ell_P$ is tangent line to $\ff$ at $P$, then clearly $I(P,\ff \cap \ell_P)=n$. Therefore, the $(\Sigma_1,P)$-order sequence is $(0,1,n)$, and then the $(\Sigma_3,P)$-order sequence is
$
(0,1,2,3,n,n+1,n+2,2n,2n+1,3n).
$
Thus Proposition \ref{SV1.7} implies that
$$
p | (n-2)(n-1)(n+1)(2n-1)(3n-1)(n-3)(3n-2)(2n+1)(2n-3)(n+2).
$$
From  Lemma \ref{classic},  we have that $p\nmid  (3n-2)(2n+1)(2n-3)(n+2)$ and the result follows.

Conversely, suppose that $p | (n-2)(n-1)(n+1)(2n-1)(3n-1)(n-3)$. If $p | (n-2)(n-1)(n+1)(2n-1)$, then Theorem \ref{gv nc} implies that $\ff$ is nonclassical w.r.t. $\Sigma_2$, and then $\ff$ is  nonclassical w.r.t. $\Sigma_3$ (cf. Lemma \ref{con imp cub}).  The case $p | (3n-1)(n-3)$  follows from Lemma \ref{p div (n-3)(3n-1)}.
\end{proof}

\begin{rem}
The restriction $p>11$ in Theorem \ref{classicalidade} cannot be dropped. To see this, consider the curve $\ff: X^{n}+Y^{n}=Z^n$ over $\overline{\mathbb{F}}_{11}$ with $n \equiv -2 \mod 11$. For $P=(u:v:1) \in \ff$ such that $uv \neq 0$, let $\G_1:\alpha X^2Z^2+\beta Y^2Z^2 -X^2Y^2=0$ be the irreducible curve as defined  in the proof of Lemma \ref{classic}. It can be checked that $\G_1$ is nonclassical w.r.t. $\Sigma_3$. That is, there exists a cubic $\cc_P$ such that $I(P,\G_1 \cap \cc_P) \geq 11$. Therefore, from Lemma \ref{trans}, $I(P,\ff \cap \cc_P) \geq 11$. In other words,  $11|n+2$,  but the curve $\ff$ is nonclassical w.r.t. $\Sigma_3$.

\end{rem}

\begin{rem}

Assume that $p>M$ is not a divisor of $n-1\geq s$, and consider the   Fermat curve $\ff: x^n+y^n+1=0$.  Since the inflection points of  $\ff$  have  $(\Sigma_1,P)$-order sequence $(0,1,n)$, it follows from Proposition \ref{SV1.7} that if $\ff$ is nonclassical w.r.t. $\Sigma_s$, then
$
p \text{ divides }\prod_{i=1}^{s}\prod_{t=-s}^{s-i}(in+t).
$

In \cite[Remark 5]{GV1}, Garcia and Voloch somewhat raised the question of whether or not the converse of this statement  holds. Theorem \ref{classicalidade} gives a negative answer: if $p|(3n-2)(2n+1)(2n-3)(n+2)$, then $\ff$ is classical w.r.t. $\Sigma_3$. 
\end{rem}


\section{$\fq$-Frobenius classicality of $\ff$ with respect to cubics}

In this section we provide the additional results that will lead us to the proof of Theorem \ref{frobenius}. Henceforth, we consider the irreducible curve $\ff:aX^n+bY^n=Z^n$, where $n>3$ and $p>7$.

\begin{lem}\label{ospart}
If $p$ divides $(n-3)(3n-1)$, then the following holds
\begin{enumerate}[\rm(a)]
\item The order sequence of $\ff$ w.r.t. $\Sigma_3$ is $(0,1,2,3,4,5,6,7,8,p^r),$ for some $r>0$

\item The curve $\ff$ is $\fq$-Frobenius nonclassical w.r.t. $\Sigma_3$ if and only if $\Phi_q(P) \in \hh_P$ for infinitely many points $P \in \ff$. 
\end{enumerate}

\end{lem}
\begin{proof}
Note that $p|(n-3)(3n-1)$ implies $p \nmid(n-1)(n-2)(n+1)(2n-1)$. Thus, from  Theorem \ref{gv nc}, the curve $\ff$ is classical w.r.t. $\Sigma_2$. That is,  $(0,1,2)$ and $(0,1,2,3,4,5)$ are the order sequences of $\ff$ w.r.t. $\Sigma_1$ and $\Sigma_2$, respectively. Then, considering the degenerated cubics, we have that the order sequence of $\ff$ w.r.t. $\Sigma_3$ is
$$
(\varepsilon_0,\ldots,\varepsilon_8, \varepsilon_9)=(0,1,2,3,4,5,6,7,\varepsilon_8, \varepsilon_9),
$$
with $\varepsilon_8 \geq 8$ and $\varepsilon_9 \geq p$ (c.f.  Proposition \ref{p div (n-3)(3n-1)}). Suppose that $\varepsilon_8 > 8$. From Theorem \ref{cor1.9}, we have that $\varepsilon_8 \geq p$. Let $P \in \ff$ be a $\Sigma_3$-ordinary point, that is, $P$ is such that $j_i(P)=\varepsilon_i$ for all $i \in \{0,\ldots,9\}$. Let $\cc_P$ be a cubic for which $I(P,\ff \cap \cc_P)=\varepsilon_8$ and let $\hh_P$ be the osculating cubic to $\ff$ at $P$. Note that $\hh_P \neq \cc_P$. Lemma \ref{trans} implies that 
$$
I(P,\hh_P \cap \cc_P) \geq p>9=\deg(\hh_P)\cdot\deg(\cc_P).
$$
Thus by Bezout's Theorem the curves $\hh_P$ and $\cc_P$ have a common component. However, from Proposition \ref{p div (n-3)(3n-1)}, the osculating cubic $\hh_P$ is irreducible. Therefore, $\hh_P = \cc_P$, a contradiction. Hence $\varepsilon_8=8$. Now it follows  from Proposition \ref{pwr} that $\varepsilon_9=p^r$ for some $r>0$.
The second assertion follows directly from the first one together with Proposition \ref{prop1}.
\end{proof}

The next result follows from \cite[Theorem 3.2]{FM}.

\begin{lem}\label{pol}
Let $K$ be an arbitrary field. Consider nonconstant polynomials $b_1(x), b_2(x) \in K[x]$, and let $l$ and $m$ be positive integers. Then $$y^l-b_1(x) \textrm{ divides } y^m-b_2(x)$$ if  and only if $l|m$ and $b_2(x)=b_1(x)^{\frac{m}{l}}$.
\end{lem}

\begin{prop}\label{frob n-3}
Suppose that $p$ divides $n-3$. The curve $\ff$ is $\fq$-Frobenius nonclassical with respect to $\Sigma_3$ if and only if
$$
n=\frac{3(p^h-1)}{p^r-1}
$$
for some $r<h$ such that  $r|h$, and $a,b \in \F_{p^r}$.
\end{prop}
\begin{proof}
By Lemma \ref{ospart} $\ff$ is $\fq$-Frobenius nonclassical w.r.t. $\Sigma_3$ if and only if $\Phi(P) \in \hh_P$ for infinitely many points $P \in \ff$, where $\hh_P$ denotes the osculating cubic to $\ff$ at $P$. Thus  Proposition \ref{p div (n-3)(3n-1)} implies that this is equivalent to
the function  
$$
ax^{n-3+3q}+by^{n-3+3q}-1
$$
being vanishing. 

Therefore,  seeing the functions  as  polynomials,  $\ff$ is $\fq$-Frobenius nonclassical if and only if 
$$
y^n-\left(\frac{1}{b}-\frac{a}{b}x^n\right) \ \textrm{ divides } \ y^{n+3(q-1)}-\left(\frac{1}{b}-\frac{a}{b}x^{n+3(q-1)}\right).
$$
By Lemma \ref{pol}, that means  $n|n+3(q-1)$ and
\begin{equation}\label{aux}
\left(\frac{1}{b}-\frac{a}{b}x^n\right)^{\frac{3(q-1)}{n}+1}=\frac{1}{b}-\frac{a}{b}x^{n+3(q-1)}.
\end{equation}

Clearly equation (\ref{aux}) implies $\frac{3(q-1)}{n}+1=p^r$ for some $r>0$, that is, $p^r-1$ divides $3(p^h-1)$.  Since  $n>3$ and $p>2$, it follows that  $r<h$ and $r|h$. It is also clear that  (\ref{aux})  gives $a,b \in \F_{p^r}$. Conversely,  the latter conditions obviously imply equation (\ref{aux}), which completes the proof.
\end{proof}

\begin{prop}\label{frob 3n-1}
Suppose that $p$ divides $3n-1$. The curve $\ff$ is $\fq$-Frobenius nonclassical with respect to $\Sigma_3$ if and only if 
$$
n=\frac{p^h-1}{3(p^r-1)}
$$ for some $r<h$ such that  $r|h$, and $a^3,b^3 \in \F_{p^r}$.
\end{prop}
\begin{proof}
As in the previous proof, by Lemma \ref{ospart} and Proposition \ref{p div (n-3)(3n-1)}, the curve $\ff$ is $\fq$-Frobenius nonclassical w.r.t. $\Sigma_3$ if and only if  the function

\begin{equation}
V:=(a^3x^{3n-1+q}+b^3y^{3n-1+q}-1)^3 +27a^3b^3x^{3n-1+q}y^{3n-1+q}
\end{equation} 
is vanishing. Therefore, the condition  $n=\frac{p^h-1}{3(p^r-1)}$, for some $r<h$ such that  $r|h$, and $a^3,b^3 \in \F_{p^r}$  implies
\begin{equation}\label{eqvi}
V=\Big((a^3x^{3n}+b^3y^{3n}-1)^3 +(3abx^{n}y^{n})^3\Big)^{p^r}.
\end{equation}
Using \eqref{eqvi} to replace $by^n$ by $1-ax^n$ yields  $V=0$, which gives the result.
Conversely, suppose $V=0$. That is, 
\begin{equation}\label{eq1}
(a^3x^{3n+q-1}+b^3y^{3n+q-1}-1)^3+27a^3b^3x^{3n+q-1}y^{3n+q-1}=(ax^n+by^n-1)h(x,y)
\end{equation}
for some $h(x,y)\in \mathbb{F}_q[x,y]\backslash \{0\}.$ Evaluating both sides of (\ref{eq1}) at $y=0$ yields $$(a^3x^{3n+q-1}-1)^3=(ax^n-1)h(x,0).$$ This implies that $ax^n-1$ divides $a^3x^{3n+q-1}-1$, and then $n\mid q-1$. Therefore, we may use (\ref{eq1}) to replace $y^n$ by $(1-ax^n)/b$, and then write
\begin{equation}\label{eq2}
\Big(a^3x^{3n+q-1}+b^3(\frac{1-ax^n}{b})^{3+\frac{q-1}{n}}-1\Big)^3=-27a^3b^3x^{3n+q-1}(\frac{1-ax^n}{b})^{3+\frac{q-1}{n}}.
\end{equation}
Since $(\frac{1-ax^n}{b})^{3+\frac{q-1}{n}}$ is a factor of both sides of (\ref{eq2}), we conclude that $3n\mid q-1$. 

Let $r$ and $t$ be integers such that $1+\frac{q-1}{3n}=p^rt$ and $p\nmid t$. From equation (\ref{eq2}), we have
\begin{equation}\label{eq3}
\Big(a^\frac{3}{p^r}x^{3nt}+b^\frac{3}{p^r}(\frac{1-ax^n}{b})^{3t}-1\Big)^3=-27(ab)^\frac{3}{p^r}x^{3nt}(\frac{1-ax^n}{b})^{3t}.
\end{equation}

Now equation (\ref{eq3}) implies that $(1-ax^n)^t$ is a factor of $a^\frac{3}{p^r}x^{3nt}-1$. Since the latter polynomial is separable, it follows that $t=1$. Hence $1+\frac{q-1}{3n}=p^r$, that is, $n=\frac{q-1}{3(p^r-1)}$. Moreover, using equation (\ref{eq3}) to replace $x^n$ by $0$ and $1/a$, we obtain $b^3\in \F_{p^r}$ and $a^3\in \F_{p^r}$, respectively. This finishes the proof.
\end{proof}

\noindent
{\bf Proof of Theorem \ref{frobenius}}: It follows directly from Theorems \ref{gv nc} and \ref{classicalidade}, and  Propositions \ref{frob n-3} and \ref{frob 3n-1}. \qed


\section{The number of rational points}

Recall that $q=p^h$, $p>11$ and $\ff:aX^n+bY^n=Z^n$ with $ab \neq 0$. As mentioned in Section \ref{intro}, if $\ff$ is $\fq$-Frobenius classical w.r.t. $\Sigma_3$, then bound (\ref{SV-sc}) for $s=3$ holds. In other words, in addition to the curves characterized in Theorem \ref{gv nc} and Theorem \ref{frobenius}, we have
\begin{equation}\label{SV-3}
N_q(\ff) \leq \frac{n}{3}(5n+q-1)-\frac{d}{3}(n-3),
\end{equation}
where $d$ is the number of $\fq$-rational points $P=(u:v:w) \in \ff$ for which $uvw=0$. Denote $m:=\gcd(n,q-1)$ and let $\ff^{'}:aX^m+bY^m=Z^m$. It is well known that $N_q(\ff^{'})=N_q(\ff)$. Therefore, we may assume that $n | q-1$. In such case, we have $N_q(\ff) \equiv d \mod n^2$, and using (\ref{SV-3}) we obtain
\begin{equation}\label{sv3}
N_q(\ff)\leq n^2 \left\lfloor\frac{5n+q-d-1}{3n}\right\rfloor+d,
\end{equation}
where $\lfloor e \rfloor$ denotes the integer part of $e$.

Note that, in such case, bound (\ref{SV-sc}) for $s=1$ and $s=2$ yields

\begin{equation}\label{sv1}
N_q(\ff) \leq n^2 \left\lfloor\frac{n+q-d-1}{2n}\right\rfloor+d
\end{equation}
and 

\begin{equation}\label{sv2}
N_q(\ff) \leq n^2 \left\lfloor\frac{2(2n+q-d-1)}{5n}\right\rfloor+d
\end{equation}
respectively. Hence bound (\ref{sv3}) is better then bounds (\ref{sv1}) and (\ref{sv2}) when, roughly, $n<\frac{q-d-1}{13}$.

\begin{ex}
Consider  the curve $\ff:X^8+Y^8+Z^8=0$ over $\F_{13^2}$. It can be checked that $N_q(\ff)=512$ and  $d=0$. Hence $\ff$ attains bound (\ref{sv3}).
\end{ex}

The possible values of $N_q(\ff)$ for the curves characterized in Theorem \ref{gv nc} are discussed in \cite{GV2}. In this section, we determine $N_q(\ff)$ for the curve $\ff$   given in Theorem \ref{frobenius}. 

\begin{thm}\label{npon1}
Suppose that $n=\frac{3(p^h-1)}{p^r-1}$ and $a,b \in \F_{p^r}$. 
\begin{itemize}

\item[(1)] If $p^r \equiv 1 \mod 3$, then
$$
N_q(\ff)=\frac{n^2}{9}\Big(N_{p^r}(\cc)-k\Big)+\frac{nk}{3},
$$
where $\cc$ is the curve $aX^3+bY^3=Z^3$ defined over  $\F_{p^r}$, and $k:=\#\{Q=(x_0:x_1:x_2) \in \cc(\F_{p^r}) \ | \ x_0x_1x_2=0\}$.

\item[(2)] If $p^r \not\equiv 1 \mod 3$, then $N_q(\ff)=\frac{n^2}{9}(p^r-2)+n.$ 
\end{itemize}
\end{thm}
\begin{proof}
The map $\rho:\ff(\fq) \lra \cc(\F_{p^r})$ given by $(x_0:x_1:x_2) \mapsto (x_0^{\frac{n}{3}}:x_1^{\frac{n}{3}}:x_2^{\frac{n}{3}})$ is clearly  well defined. Since $x \mapsto x^{\frac{n}{3}}$ is  the norm function  of $\fq$ onto $\F_{p^r}$, we have  
$\ff(\fq)= \bigcup_{Q \in \cc(\F_{p^r})} \rho^{-1}(Q).$
Thus, setting $k:=\#\{Q=(x_0:x_1:x_2) \in \cc(\F_{p^r}) \ | \ x_0x_1x_2=0\}$, we arrive at
$$
N_q(\ff)=\left(\frac{n}{3}\right)^2\big(N_{p^r}(\cc)-k\big)+\frac{n}{3}k,
$$
which proves the first assertion.

Now note that in the case $p^r \not\equiv 1 \mod 3$,  the map $\alpha \mapsto \alpha^3$ permutes $\F_{p^r}$, and then $k=3$.  Moreover, in this case, $N_{p^r}(\cc)=p^r+1$, which finishes the proof. 
\end{proof}

\begin{thm}\label{npon2}
If $n=\frac{p^h-1}{3(p^r-1)}$ and $a^3,b^3 \in \F_{p^r}$, then

\begin{equation*}
N_q(\ff)= \begin{cases} 
3n+n^2\big(p^r-2\big), & \text{ if }\quad p^r \equiv 1 \mod 3\\
3n+n^2p^r, & \text{  otherwise.}
\end{cases} 
\end{equation*}
\end{thm}
\begin{proof}
Since $a^3, b^3 \in \F_{p^r}$, we may assume that $\ff$ is defined by $X^n+Y^n+Z^n=0$. Setting $m=h/r$, we obtain $\frac{(p^r)^m-1}{n(p^r-1)}=3$. Thus the result follows from a direct application of Theorem \ref{ferko}, observing that:
\begin{itemize}
\item if $p^r \not\equiv 1 \mod 3$, then $t=2$ and $l=3$.
\item if $p^r \equiv 1 \mod 3$, then $t=l=1$.
\end{itemize}
\end{proof}

Nor surprisingly, using the two previous theorems, one can find examples of curves for which the upper bound  (\ref{sv3}) fails.

\begin{ex}
Consider the curve $\ff:X^{294}+Y^{294}=Z^{294}$ over $\F_{97^2}$. Here $\ff$ has degree $n=3\frac{97^2-1}{97-1}$. The number of $\F_{97}$-rational points of the cubic $\cc:=X^3+Y^3=Z^3$ is $N_{97}(\cc)=117$. Thus Theorem \ref{npon1} yields $N_{97^2}(\ff)=1038114$. Since, in this case, $d=3n=882$, it follows that $N_{97^2}(\ff)$ exceeds the upper bound in (\ref{sv3}).
\end{ex}

\begin{ex}
Let $\ff$ be the curve $X^8+Y^8+Z^8=0$ over $\F_{23^2}$. Since $\ff$  has degree $n=\frac{23^2-1}{3(23-1)}$, it follows from Theorem \ref{npon2} that $N_{23^2}(\ff)=1496$. One  can be check  that  $d=24$, and then $N_{23^2}(\ff)$ exceeds the upper bound in (\ref{sv3}).
\end{ex}

\appendix

\section{Some irreducible curves}
\begin{lem}\label{polirr}
If $p > 3$, then the following curves are irreducible over $\fqc$:
\begin{itemize}
\item  $\tilde{\G_1}:X^2Z^2+ Y^2Z^2 -X^2Y^2=0$ 
\item  $\tilde{\G_2}:Y^2Z^2+X^2Z^2+X^2Y^2-2XYZ(X+ Y+Z)=0$
\item  $\tilde{\G_3}:X^6+Y^6+Z^6-2(X^3Y^3+X^3Z^3+Y^3Z^3)=0$
\item  $\tilde{\G_4}:(X^2+Y^2-Z^2)^3+27X^2Y^2Z^2=0$.
\end{itemize}
\end{lem}
\begin{proof}
The irreducibility of $\tilde{\G_1}$ and $\tilde{\G_2}$ follows from \cite[Lemma A.1]{AB}. 

The proofs for the irreducibility of $\tilde{\G_3}$ and $\tilde{\G_4}$ are similar. Thus we will prove the latter case only. 

  For the curve $\tilde{\G_4}$, one can readily check that its set of singular points is given by $\mathscr{C} \cup \mathscr{N}$ where

$$
\mathscr{C}=\{(0:1:1),(0:-1:1),(1:0:1),(-1:0:1),(i:1:0),(-i:1:0) \ \ | \ \ i^2=-1\},
$$
 and
$$
\mathscr{N}=\{(i:i:1),(i:-i:1),(-i:-i:1),(-i:i:1) \ \ | \ \ i^2=-1\}
$$
are  the sets of cusps and nodes of $\tilde{\G_4}$, respectively. We proceed to show that $\tilde{\G_3}$ has no component of degree $\leq 3$.
 Note that the lines $x=0,y=0$ and $z=0$ intersect  $\tilde{\G_4}$ in  pairs of cusps $\{P_1,P_2\}$, $\{Q_1,Q_2\}$ and $\{R_1,R_2\}$, whose union is $\mathscr{C}$. Therefore, any component of $\tilde{\G_4}$  must contain at least $3$ points $P_i$, $Q_j$, $R_k$ for some $(i,j,k)\in \{1,2\}^3$. Since no choice of $3$ such points will be collinear, the curve has no linear components. 
 
 Now assume $\tilde{\G_4}=\cc \cup \qq$, where $\cc$ is a smooth conic and $\qq$ is an irreducible quartic. Since the quartic $\qq$ has at most $3$ singularities,  $\cc$ and  $\qq$ must intersect in at least $3$ distinct cusps in $\mathscr{C}$. Thus, by B\'ezout's theorem, $\cc$ and  $\qq$ intersect in at most $5$ distinct points, and then  $\tilde{\G_4}$ has at most $8$ singular points. This contradicts $\#(\mathscr{C} \cup \mathscr{N})=10$.
 
Suppose $\tilde{\G_4}$ is the union of $3$ distinct smooth conics. By  B\'ezout's theorem, the intersection of these conics yields $12$ (counted with multiplicities) singular points of  $\tilde{\G_4}$. Since all points of 
$\mathscr{C}$ are cusps, its $6$ points must be counted at least twice. Thus all singular points of  $\tilde{\G_4}$  lie in $\mathscr{C}$, a contradiction.

 Finally, suppose that $\tilde{\G_4}$ is the union of $2$ irreducible cubics. In the worst case  scenario, each cubic has one cusp. Thus, similarly to the previous cases, the remaining  $4$ cusps will give rise to a counting contradiction.  Therefore $\tilde{\G_4}$ is irreducible.

\end{proof}

\begin{lem}\label{poliosc}
Suppose $p\in \{13,17\}$, and let $\tilde{\G_i}$ be the curves given by Lemma \ref{polirr}.   For   each $i\in \{1,\dots, 4\}$, there exist a nonsigular point  $\tilde{P_i}=(s_i:s_i:1) \in \tilde{\G_i}$ and a nonsingular cubic $\tilde{\mathcal{C}_i}$ such that 
$$I(\tilde{P_i},\tilde{\G_i} \cap \tilde{\mathcal{C}_i})\in \{10,12\}$$

\end{lem}
\begin{proof}
Due to its simple but computational nature, our proof will be limited to presenting  each point $\tilde{P_i}=(s_i,s_i)$ and the corresponding cubic $\mathcal{C}_i$ in affine coordinates.
\begin{itemize}
\item  $\tilde{P_1}=(s,s)$ where $s^2=2$, and\\
$\mathcal{\tilde{C}}_1: x^3 + 1677x^2y -1194sx^2 + 1677xy^2 -1848sxy + 996x + y^3 -1194sy^2 + 996y -232s=0.$
\item   $\tilde{P_2}=(4,4)$,  and\\
$\mathcal{\tilde{C}}_2:  x^3 + 543x^2y-672x^2 + 543xy^2 + 2112xy-8448x + y^3 -672y^2 -8448y -14336=0$

\item   $\tilde{P_3}=(s,s)$ where $s^3=1/4$, and\\
$\mathcal{\tilde{C}}_3: 13x^3 + 27x^2y - 27sx^2 + 27xy^2 - 42sxy + 13y^3 - 27sy^2 + 4=0$

\item   $\tilde{P_4}=(s,s)$ where $s^2=1/8$, and\\
$\mathcal{\tilde{C}}_4:  532x^3 + 804x^2y - 6216sx^2 + 804xy^2 - 9120sxy + 2841x + 532y^3 -  6216sy^2 + 2841y - 3322s=0$
\end{itemize}

\end{proof}

\section{Frobenius nonclassicality of $aX^{q-1}+(1-a)Y^{q-1}=Z^{q-1}$ with respect to conics}

\begin{thm}
Suppose $p>5$  divides $n+1$. Then the Fermat curve $\ff:aX^{n}+bY^{n}=Z^{n}$  is $\fq$-Frobenius nonclassical with respect to $\Sigma_2$ if and only if $a+b=1$ and $n=q-1$.
\end{thm}
\begin{proof} Set  $x=X/Z$ and $y=Y/Z$, and  let  $\ff({\overline{\mathbb{F}}_q}):={\overline{\mathbb{F}}_q}(x,y)$ be the function field of 
$\ff$.  If $\ff$ is $\fq$-Frobenius nonclassical, then it follows from the proof  of (\cite[Theorem 3]{GV2}) that
the function $$G= x^qy^q-ax^{n+1}y^q-by^{n+1}x^q,$$ seen as a polynomial, must be identically zero. That
is, $n+1=q$ and $a+b=1$. 

Conversely,  suppose that $\ff$ is given by $aX^{q-1}+(1-a)Y^{q-1}=Z^{q-1}$.
Since $p \nmid n-1$, it follows that $\ff$ is classical w.r.t. $\Sigma_1$ \cite[Corollary 2.2]{Pa}. Thus the order sequence of $\ff$ w.r.t. $\Sigma_2$ is $(0,1,2,3,4,\varepsilon)$ with  $\varepsilon > 5$ (\cite[Theorem 3]{GV1}). 
Now, for $P=(u:v:1) \in \ff$, with $uv \neq 0$, we claim that the osculating conic to $\ff$ at $P$ has affine equation  given by
$$\cc_P: (au)^qY+\big((1-a)v\big)^qX-XY=0.$$
First note that   $h(x,y):=(ax^{q-1}+(1-a)y^{q-1}-1)xy=\big(ax\big)^qy+\big((1-a)y\big)^qx-xy=0.$
Setting $$g(x,y):= (au)^qy+\big((1-a)v\big)^qx-xy,$$ it follows that
 \begin{eqnarray}
g(x,y)&=&g(x,y)-h(x,y) \nonumber\\
      &=& (au-ax)^qy+\big((1-a)v-(1-a)y\big)^qx \nonumber,
\end{eqnarray}
  and then $v_P(g(x,y)) \geq q >5$. That is, $(au)^qY+\big((1-a)v\big)^qX-XY=0$ is the osculating conic to $\ff$ at $P=(u:v:1)$.
  
  Let $\Phi: \ff \rightarrow  \ff $ be the $\fq$-Frobenius map.  Since $\ff$ has order sequence $(0,1,2,3,4,\varepsilon)$, by Proposition \ref{prop1} it  is $\fq$-Frobenius nonclassical w.r.t. $\Sigma_2$ if and only  if the function 
  $$(ax)^qy^q+\big((1-a)y\big)^qx^q-x^qy^q$$
  is vanishing. Thus the result follows.
\end{proof}

\section{Proof of Lemma \ref{trans}}

\begin{proof}
Let $f$, $g$, $h$ be local equations  of $C$, $D_1$, $D_2$ in $\mathcal{O}_{S,P}$, respectively. Then $I(D_1. D_2,P)=\dim_K\mathcal{O}_{S,P}/(g,h)$.
Since $\mathcal{O}_{C,P}=\mathcal{O}_{S,P}/(f)$, the map $\mathcal{O}_{S,P}/(g,h) \rightarrow \mathcal{O}_{C,P}/(\overline{g},\overline{h})$,  where 
$\overline{g}$ and $\overline{h}$ are the images of $f$ and $g$ in $\mathcal{O}_{C,P}$,  is surjective. Hence, 
\begin{equation} \label{Hom1}
I(D_1. D_2,P) \geq \dim_K\mathcal{O}_{C,P}/(\overline{g},\overline{h}).
\end{equation}
On the other hand,
\begin{align*}
I(C . D_1,P) &=\dim_K  \mathcal{O}_{S,P}/(f,g)\\
& = \dim_K  \mathcal{O}_{C,P}/(\overline{g})\\
& = v_P(\overline{g}) \numberthis \label{Hom2}
\end{align*}
where $v_P$ is the valuation of  $\mathcal{O}_{C,P}$, and also
\begin{equation} \label{Hom3}
I(C. D_2,P) =v_P(\overline{h}).
\end{equation}

Let $t \in \mathcal{O}_{C,P}$ be a local parameter. Then $\hat{\mathcal{O}}_{C,P}\cong K[[t]]$. Since $\dim_K  K[[t]]/(\overline{g})=v_P(\overline{g})$ and $\dim K[[t]]/(\overline{h})=v_P(\overline{h})$, we have $\dim_K  K[[t]]/(\overline{g},\overline{h})=min \{v_P(\overline{g}),v_P(\overline{h})\}$. Thus
\begin{equation}\label{Hom4}
\dim_K  \mathcal{O}_{C,P}/(\overline{g},\overline{h})\geq min \{v_P(\overline{g}),v_P(\overline{h})\},
\end{equation}
because  $\mathcal{O}_{C,P}/(\overline{g},\overline{h}) \rightarrow K[[t]]/(\overline{g},\overline{h})$ is surjective.
Therefore, from \ref{Hom1},\ref{Hom2},\ref{Hom3},\ref{Hom4}, we have 
$$I(D_1. D_2,P) \geq min \{I(C . D_1,P),I(C. D_2,P)\}.$$
\end{proof}

\section*{Acknowledgments}
The first author was supported by FAPESP-Brazil, grant 2013/00564-1. The second author is thankful for the great support received from the ICTP-Trieste, as a visitor, during the last two months of this project.


\printindex

\end{document}